\documentclass[11pt,reqno]{amsart}
	
	\usepackage{amssymb,amsmath,amsthm,amscd,latexsym,amsfonts}
	\usepackage{mathtools}
	\usepackage[T1]{fontenc}
	\usepackage{graphicx,epstopdf}
	\usepackage{graphicx}
	\usepackage{xcolor}
	\usepackage{cite}
	\usepackage{float}
	\usepackage[a4paper,top=3.1cm, bottom=3.1cm, left=3.3cm, right=3.3cm]{geometry}
	
	\newtheorem{thm}{Theorem}
	
	\newtheorem{lemma}{Lemma}
	
	\newtheorem{rk}{Remark}
	\newtheorem{cor}{Corollary}
	
	\newtheorem{ex}{Example}

	\numberwithin{equation}{section} 

	\newcommand{\bea}{\begin{eqnarray}}
		\newcommand{\eea}{\end{eqnarray}}

\begin{document}
		\title [QSO splitted to linear operators]
		{A many-loci system with evolution characterized by  uncountable linear operators}
		
		\author{B.A.Omirov, U.A. Rozikov}
		\address{ B.A. Omirov$^{a,b}$\begin{itemize}
\item[$^a$] Suzhou Research Institute, Harbin Institute of Technologies, Harbin  215104, Suzhou, P.R. China;
\item[$^b$] Institute for Advanced Study in Mathematics,
Harbin Institute of Technologies, Harbin 150001, P.R. China;
		\end{itemize}}
		\email{omirovb@mail.ru}
\address{ U.A. Rozikov$^{c,d,f}$\begin{itemize}
				\item[$^c$] V.I.Romanovskiy Institute of Mathematics,  Uzbekistan Academy of Sciences, 9, Universitet str., 100174, Tashkent, Uzbekistan;
				\item[$^d$]  National University of Uzbekistan,  4, Universitet str., 100174, Tashkent, Uzbekistan.
				\item[$^f$] Karshi State University, 17, Kuchabag str., 180119, Karshi, Uzbekistan.
		\end{itemize}}
		\email{rozikovu@yandex.ru}
		
		\begin{abstract} This paper investigates the evolution of a multi-locus biological system. The evolution of such a system is described by a quadratic stochastic operator (QSO) defined on a simplex. We demonstrate that this QSO can be decomposed into an infinite series of linear operators, each of which maps certain invariant subsets of the simplex to themselves. Furthermore, the entire simplex is the union of these invariant subsets, enabling analytical examination of the dynamical systems produced by the QSO. Finding all limit points of the dynamical system generated by the QSO in terms of the limit points of the linear operators, we provide a comprehensive characterization of the many-loci population dynamics.
		\end{abstract}
		\maketitle


		\section{Introduction}
	Mathematical investigations are important to a better understanding of
	living populations at all levels (see \cite{E}, \cite{GPR}, \cite{Rpd} and references therein).	
	
	\subsection{QSO and the main problem}  The quadratic stochastic operator (QSO), which maps the simplex to itself, is often used as an evolutionary operator in population dynamics. Let us recall definition, given in \cite{Rpd}, of a such QSO first:
	
Consider a population consisting of $m$ species. Let $x^{0} =(x_{1}^{0},\dots,x_{m}^{0})$ be the probability distribution  of these species in the initial generation and $P_{ij,k}$ be the probability	that individuals in the $i$-th and $j$-th species interbreed to produce an individual $k$.

Assume the population is free, i.e., there is no difference in sex and in any generation the "parents" $ij$ are independent. Then the probability distribution $x'= (x_{1}',\dots,x_{m}')$ of the species in the first generation can be determined by the total probability as
	\begin{equation}\label{0a1}
		x'_k=  \mathop {\sum}
		\limits^{m}_{i,j=1}P_{ij,k}x^{0}_{i}x^{0}_{j} ,\,\,\,\,k= 1,\dots,m.
	\end{equation}
	The association $x^{0}\rightarrow x'$
	defines a map $V$ called the \textit{evolution operator}.
	
	The states of population described by the following discrete-time dynamical
	system \index{discrete-time dynamical system}

	\begin{equation}\label{0a2}
		x^{0},\ \ x^{(k)}= V^k(x^0), \, k\geq 1
	\end{equation}
	where $V^{k}(x)=V(V^{k-1}(x)).$
	
One of {\bf main problems} in a given dynamical system is to describe the limit points of $\{x^{(n)}\}_{n=0}^\infty$ for an arbitrary initial $x^{0}$.
	
In \cite{GMR}, known results on the problem, as well as several open problems related to the theory of QSOs, are presented (see also \cite{L}, \cite{MG}, \cite{Rpd} for a more detailed discussion of the theory of QSOs). Since QSOs are nonlinear, this problem has been studied for a few classes of QSOs. Therefore, it is important to identify a specific class of QSOs for which we can analytically study the dynamical systems generated by these operators.

In this paper, we investigate the evolution of many-loci systems, which is represented by a specific QSO. The main result of our study is the identification of a class of such specific QSOs, each of which can be reduced to infinitely many linear operators that map invariant subsets of the simplex to themselves. Since linear operators are well understood, we can fully characterize the set of all limit points of the dynamical system generated by our QSOs in terms of the limit points of these linear operators. Thus, we have thoroughly described the set of all limit points and provided biological interpretations of our results as they apply to many-loci systems.
	
\subsection{Biological background and motivations}

Let us give some necessary notations from biology, which can be found, for instance, in \cite{Gr} and \cite{St}. A locus (plural: loci) is a specific, fixed position on a chromosome where a particular gene or genetic marker is situated. This can be seen as an address on a chromosome, identifying the precise location of a gene or other genetic sequence.

Each chromosome, which is a long, continuous thread of DNA containing numerous genes, has many-loci, each corresponding to a different gene. Thus, every gene occupies a unique position or locus on a chromosome. For instance, humans have 23 pairs of chromosomes. In a haploid set, which includes one chromosome from each pair (totaling 23 chromosomes), the number of protein-coding genes is estimated to range between 19,000 and 20,000.

An evolution operator describes how the genetic composition of a population changes over time due to various evolutionary forces such as mutation, mating, selection, recombination and genetic drift.

For a many-loci system, an evolution operator needs to account for the interactions and contributions of multiple loci. One common way to represent this is through a transition operator (in particular, a matrix) or an integral operator in the context of population genetics models.

\subsubsection{Motivation: Examples of many-loci systems in real-life}

Many-loci systems refer to genetic scenarios where multiple genes (loci) influence a single trait. These systems are common in various biological contexts. The real-life examples are human height; skin and eye color; disease resistance in plants; quantitative trait loci in agriculture;
complex diseases in humans such as diabetes, heart disease, and schizophrenia, are influenced by the interaction of multiple genetic loci and environmental factors (see, for example, \cite{E}, \cite{Fl}, \cite{Hi}, \cite{Stu}, \cite{Vi}).

For two real-life examples of many-loci systems we give their evolution operator:

1. {\it Human height} is a classic example of a polygenic trait, influenced by many genetic loci. Genome-wide association studies (GWAS) have identified hundreds of loci associated with height (see \cite{Ha}).

To write an evolution operator one has to take into account the allele frequencies at loci change over generations due to natural selection, mutation (introducing new alleles), recombination (reshuffling genetic material), and genetic drift (random changes in allele frequencies). The evolution operator in this case would integrate these effects to model changes in the distribution of height in the population over time.

Let us consider a simplified example of the evolution operator for human height. Suppose we have three loci $(A, B, C)$ contributing to height.

The evolution operator for allele frequency $p_A$ at locus $A$ can be given as follows
$$ p_A' = \frac{p_A (1 + s_A \cdot w_A) (1 - \mu) + (1 - p_A) \cdot \mu + r \cdot (p_B \cdot p_C(1 - p_A) - p_A ) + \text{drift term}}{(1 + \bar{w})},$$
where

-  $ s_A  $ is the selection coefficient for allele $A$,

-  $ w_A  $ is the fitness contribution of allele $A$,

-  $ \mu  $ is the mutation rate,

-  $ r  $ is the recombination rate,

-  $ \bar{w}  $ is the mean fitness of the population.

Similarly, one can write formula for $ p_B'$ and $ p_C'$. Therefore, we obtain a non-linear operator $V: (p_A, p_B, p_C)\to
(p'_A, p'_B, p'_C)$ which is called {\it an evolution operator of human height}.

This simplified example illustrates how an evolution operator integrates multiple evolutionary forces to describe changes in allele frequencies in a population. In practice, such models involve more complex interactions and greater number of loci, especially for traits like human height, which are influenced by numerous genetic factors.\\

2. {\it Human susceptibility to complex diseases.}
Complex diseases like diabetes, heart disease, and schizophrenia are influenced by multiple genetic loci. GWAS have identified many loci associated with these diseases (see \cite{Ot}, \cite{Vi}, \cite{Wr} and the references therein).

 In the context of human populations, the evolution operator would integrate the effects of natural selection (e.g., differential survival and reproduction based on disease susceptibility), mutation rates (introducing new alleles associated with disease risk), recombination and genetic drift. This operator helps in modeling the changes in disease prevalence and genetic risk factors over time.

Mathematical modeling of human susceptibility to complex diseases requires an understanding of the interplay between genetic, environmental, and lifestyle factors. Let $\beta_i$ represent the effect size of the $i$-th genetic variant and $G_i$ denote the genotype $(0, 1$, or $2)$ for the $ i$-th variant. Then the probability of disease is defined by
$$
P(D) = \Phi^{-1}\left( \sum_{i=1}^{n} \beta_i G_i + \epsilon \right),$$
where $ \Phi^{-1}$ is the inverse of the cumulative distribution function of the standard normal distribution and $\epsilon$ represents environmental and residual effects.

The exposure effects are given by
$$E = \sum_{j=1}^{m} \gamma_j E_j,$$
where $\gamma_j$ represents the effect size of the $j$-th environmental factor, and $E_j$ is the exposure level.

Comprehensive risk model is represented as follows
$$R = \alpha + \sum_{i=1}^{n} \beta_i G_i + \sum_{j=1}^{m} \gamma_j E_j + \epsilon,$$
where $R$ is the overall risk score and $\alpha$ is a baseline risk term.

The evolution operator describes how the risk of disease changes over time, incorporating the dynamic nature of genetic, environmental, and lifestyle factors. This can be modeled using a set of differential equations or a Markov process.

Introduce notations

 $R(t)$ - the risk at time $t$,

 $G(t)$ - represents the genetic component,

 $ E(t)$ - the environmental exposures,

 $L(t)$ - lifestyle factors.

Then a continuous time model is given by differential equation
$$
\frac{dR(t)}{dt} = \beta G(t) + \gamma E(t) + \delta L(t) - \mu R(t),$$
where $\beta$, $\gamma$, $\delta$, and $\mu$ are parameters that quantify the effects of genetic, environmental, and lifestyle factors, and the decay rate of risk, respectively.

A discrete time model is given by formula
$$
P(R_{t+1} | R_t, G_t, E_t, L_t) = \sum_{i=1}^{N} P(R_{t+1} | R_t = i) P(G_t, E_t, L_t | R_t = i),
$$
where $P(R_{t+1} | R_t)$ is the transition probability from state $ R_t $ to $ R_{t+1}$, given genetic, environmental and lifestyle states.

 The evolution operator in these models captures how the risk changes, influenced by the various factors contributing to the disease.

\subsubsection{Two loci system}

 In \cite[page 68]{E} a population is considered assuming viability selection, random mating and discrete non-overlapping generations. It consists of two loci $A$ (with alleles $A_1$, $A_2$) and $B$ (with alleles $B_1$, $B_2$). In this case there are four gametes: $A_1B_1$, $A_1B_2$, $A_2B_1$, and $A_2B_2$. Denote the frequencies of these gametes by $x$, $y$, $u$ and $v,$ respectively.

 Recall that $(m-1)$-dimensional simplex is defined as
 $$ S^{m-1}=\{x=(x_1,...,x_m)\in \mathbb R^m \ : \ x_i\geq 0, \
 \sum^m_{i=1}x_i=1 \}.$$

Thus, the vector $(x, y, u, v)\in S^3$ can be considered as a state of the system and therefore, one takes it as a probability distribution on the set of gametes.

In \cite[Section 2.10]{E} the frequencies $(x', y', u', v')$ of the next generation are defined as
\begin{equation}\label{yq}	
	\begin{array}{llll}
		x'=x+a\cdot (yu-xv),\\[2mm]
		y'=y-a\cdot (yu-xv),\\[2mm]
		u'=u-b\cdot (yu-xv),\\[2mm]
		v'=v+b\cdot (yu-xv),
\end{array}\end{equation}
with $a,b\in [0,1]$.
Dynamics of operator (\ref{yq}) is completely described in \cite{DR}. Moreover, dynamical systems generated by a permuted version of this operator are well studied in \cite{RS}.

\subsubsection{Many-loci system}

Following \cite[page 245]{E} we present evolution operator of many-loci system.  This operator generalizes (\ref{yq}) and it depend on both genotypic fitnesses and the recombination pattern between loci. Recombination refers the process by which genetic material is shuffled during the formation of gametes (i.e., the reproductive cells that carry half the genetic information from an organism to its offspring).

Let $ i$ and $j $ be two parental gametes involved in recombination.
Each gamete contains a haploid set of chromosomes (i.e., one set of chromosomes, half of the genetic material necessary to form a complete organism when combined with another gamete).  During the formation of gametes recombination occurs, involving the exchange of genetic material between homologous chromosomes. This results in new combinations of alleles (variants of genes) and increases genetic diversity in the offspring.

After recombination between gametes $i$ and $j$, two new gametes are formed. These new gametes have genetic material that is a mix of the original parental gametes.

Now,  randomly choose one of the two gametes formed by recombination. Let $P_{ij,h}$ denote the probability that the randomly chosen gamete (from the two formed) is gamete $h$.

If $x_i, x_i^{\prime}$ represent the frequencies of gamete $i$ in consecutive generations, then
\begin{equation}\label{eva}
\bar{w} x_i^{\prime}=w_i x_i-\sum_{{h,j:\atop i, j \neq h}} w_{i j}P_{ij,h} x_i x_j +\sum_{{h,j:\atop h, j \neq i}} w_{h j} P_{hj, i}x_h x_j.
\end{equation}
Here, $\bar w$ denotes the mean fitness, $w_i$ the marginal fitness of gamete $i$ and $w_{ij}$ is fitness of the pair $ij$.

In general, studying dynamical systems generated by the nonlinear operator (\ref{eva}) is challenging. In this paper, we identify a specific class of operators (\ref{eva}) and provide a complete characterization of the dynamical systems generated by each operator in this class.

\subsection{Our model of many-loci system} The above considered  biological models motivate us to consider general form of operator (\ref{yq}). Namely, we define operator $$W:x=(x_1,\dots, x_{2m})\in \mathbb R^{2m}\to x'=W(x)=(x'_1,\dots, x'_{2m})\in \mathbb R^{2m}$$ as
\begin{equation}\label{uyq}	
W:	\left\{\begin{array}{llllll}
x'_{2i-1}&=&x_{2i-1}+\sum\limits_{j=1}^{m} a_{ij}\left(x_{2i}x_{2j-1}-x_{2i-1}x_{2j}\right),\\[2mm]
x'_{2i}&=&x_{2i}-\sum\limits_{j=1}^{m} a_{ij}\left(x_{2i}x_{2j-1}-x_{2i-1}x_{2j}\right),
\end{array}\right. \end{equation}
where $a_{ij}\in [0,1]$ for  $1\leq i,j \leq m$.

Since for $m=1$ the operator $W$ is identity map of $\mathbb{R}^2$, we shall consider the case $m\geq 2$.

\begin{lemma}\label{su} The quadratic operator (\ref{uyq}), $W$, maps $S^{2m-1}$ to itself.
\end{lemma}
\begin{proof} Let $x=(x_1,\dots, x_{2m})\in S^{2m-1}$, we show that $x'=W(x)=(x'_1,\dots, x'_{2m})\in S^{2m-1}$. It is easy to see that $\sum\limits_{i=1}^{2m} x'_i=1$. It remains to show that each coordinate $x'_{2i-1}$, $x'_{2i}$, $i=1, \dots, m$ is non-negative. Taking into account $1=\sum\limits_{i=1}^{2m} x_i$ and $a_{ij}\in [0,1]$, we deduce
$$\begin{array}{llll}
x'_{2i-1}&=&x_{2i-1}+\sum\limits_{j=1}^{m}a_{ij}(x_{2i}x_{2j-1}-
x_{2i-1}x_{2j})\\[5mm]
&=&x_{2i-1}\Big(1-\sum\limits_{j=1}^{m}a_{ij}x_{2j}\Big)+
x_{2i}\sum\limits_{j=1}^{m} a_{ij}x_{2j-1}\\[5mm]
&\geq& x_{2i-1}\Big(1-\sum\limits_{j=1}^{m} x_{2j}\Big)+x_{2i}\sum\limits_{j=1}^{m} a_{ij}x_{2j-1}\\[5mm]
&=&x_{2i-1}\sum\limits_{j=1}^{m} x_{2j-1}+x_{2i}\sum\limits_{j=1}^{m} a_{ij}x_{2j-1}\geq 0.\\[5mm]
\end{array}$$

Similarly, for $x'_{2i}$ we derive
$$x'_{2i}\geq x_{2i}\sum_{j=1}^{m} x_{2j}+x_{2i-1}\sum_{j=1}^{m} a_{ij}x_{2j}\geq 0.$$
\end{proof}

Recall that a quadratic stochastic operator $V: S^{m-1}\to S^{m-1}$ (further denoted by QSO) is defined by
\begin{equation}\label{qso}
	V: \quad x_k'=\sum\limits_{i,j=1}^mP_{ij,k}x_ix_j,
	\end{equation}
where
\begin{equation}\label{cub}
	\sum\limits_{k=1}^mP_{ij,k}=1 \quad \mbox{and} \quad P_{ij,k}\geq 0 \quad \mbox{for any} \quad 1\leq i,j,k \leq m.
\end{equation}

The operator (\ref{qso}) has not been studied in general, although some large classes of QSOs have been explored (see, for example, \cite{L}, \cite{Rpd} and references therein). However, the operator (\ref{uyq}) has not yet been investigated.

\begin{rk} From the proof of Lemma \ref{su} it follows that  operator (\ref{uyq}) is a QSO, because of it can be rewritten  as
	\begin{equation}\label{yqs}	
		W:	\left\{\begin{array}{llllll}
			x'_{2i-1}&=&x_{2i-1}\sum\limits_{j=1}^{m}\Big(x_{2j-1}+(1- a_{ij})x_{2j}\Big)+ x_{2i}\sum\limits_{j=1}^{m}a_{ij}x_{2j-1},\\[3mm]
		x'_{2i}&=&x_{2i}\sum\limits_{j=1}^{m}\Big(x_{2j}+(1- a_{ij})x_{2j-1}\Big)+ x_{2i-1}\sum\limits_{j=1}^{m}a_{ij}x_{2j}.\\[3mm]
	\end{array}\right. \end{equation}
\end{rk}

Denote by $\mathbb P=(P_{ij,k})_{i,j,k=1}^{2m-1}$ a cubic matrix satisfying (\ref{cub}). Let $\mathbb P_k=(P_{ij,k})_{i,j=1}^{2m-1}$ be $k$-th level (a square matrix) of the cubic matrix.
	
For example, if $m=2$, then for the operator (\ref{yqs}), the corresponding cubic matrix can be rewritten as
$$\mathbb P=\left(\mathbb P_1, \mathbb P_2, \mathbb P_3, \mathbb P_4\right),$$
with
$$\mathbb P_1=\left(\begin{array}{cccc}
	1&{1\over 2}&{1\over 2}&{1-a_{12}\over 2}\\[2mm]
	{1\over 2}&0&{a_{12}\over 2}&0\\[2mm]
	{1\over 2}&{a_{12}\over 2}&0&0\\[2mm]
		{1-a_{12}\over 2}&0&0&0\\[2mm]
	\end{array}
	\right), \ \  \mathbb P_2=\left(\begin{array}{cccc}
		0&{1\over 2}&0&{a_{12}\over 2}\\[2mm]
		{1\over 2}&1&{1-a_{12}\over 2}&{1\over 2}\\[2mm]
		0&{1-a_{12}\over 2}&0&0\\[2mm]
		{a_{12}\over 2}&{1\over 2}&0&0\\[2mm]
	\end{array}
	\right),$$
	$$\mathbb P_3=\left(\begin{array}{cccc}
		0&0&{1\over 2}&{a_{21}\over 2}\\[2mm]
		0&0&{1-a_{21}\over 2}&0\\[2mm]
		{1\over 2}&{1-a_{21}\over 2}&1&{1\over 2}\\[2mm]
		{a_{21}\over 2}&0&{1\over 2}&0\\[2mm]
	\end{array}
	\right), \ \ \mathbb P_4=\left(\begin{array}{cccc}
		0&0&0&{1-a_{21}\over 2}\\[2mm]
		0&0&{a_{21}\over 2}&{1\over 2}\\[2mm]
		0&{a_{21}\over 2}&0&{1\over 2}\\[2mm]
		{1-a_{21}\over 2}&{1\over 2}&{1\over 2}&1\\[2mm]
	\end{array}
	\right).$$
	
	\section{The set of limit points.}
	
Our main problem is to describe the set of limit points of the sequence  $x^{(n)}=W^{n}(x^{(0)}), n=0,1,2,..$ for each initial point $x^{(0)}\in S^{2m-1}$.

In general, when a dynamical system is generated by a nonlinear operator, finding complete solution of the main problem may be very difficult. However, in  this paper, we provide a complete solution to the main problem for nonlinear operator (\ref{uyq}).

Note that the operator (\ref{uyq}) in the case of $a_{ij}=0$ for all $i,j$ is identity map, that is why we shall not consider this case. Moreover, if $a_{i_0j}=0$ for some $i_0$ and all $j$ ($j\ne i_0$),
then we obtain $x_{2i_0-1}'=x_{2i_0-1}$ and $x_{2i_0}'=x_{2i_0}$. Thus, operator (\ref{uyq}) with the aforementioned conditions on its coefficients is simpler than the operator which satisfies the following condition:
\begin{equation}\label{c}
\sum_{{j=1\atop j\ne i}}^ma_{ij}>0, \ \ \mbox{for all} \ \ i=1,\dots,m.
\end{equation}

To avoid dealing with special cases, we will consider the operator (\ref{uyq}) with the condition (\ref{c}) on its coefficients.

For an arbitrary element $c=(c_1, \dots, c_m)$ of $S^{m-1}$ we define
$$I_c:=\{x=(x_1,...,x_{2m})\in S^{2m-1} \ : \ x_{2i-1}+x_{2i}=c_i, \ i=1, \dots, m\}.$$

From (\ref{uyq}) we derive that $x'_{2i-1}+x'_{2i}=x_{2i-1}+x_{2i}$ for any $1\leq i \leq m.$ Therefore, for any $c\in S^{m-1}$ the set
$I_c$ is invariant with respect to $W$, i.e., $W(I_c)\subseteq I_{c}.$

Note that
\begin{equation}\label{ss}
	S^{2m-1}=\bigcup_{c\in S^{m-1}} I_c.
\end{equation}
Since $I_c$ is an invariant with respect to $W$, it suffices to study limit points of the operator $W$ on sets $I_c$ for each $c\in S^{m-1}$ separately.

\begin{lemma} \label{lem2} For each $c\in S^{m-1}$ the operator $W_c:=W_{|I_c}$ is a linear map given by
\begin{equation}\label{chiz}	
	x'_{2i-1}=\Big(1-\sum_{{j=1\atop j\ne i}}^{m} a_{ij}c_j\Big) x_{2i-1}+c_i\sum_{{j=1\atop j\ne i}}^{m}a_{ij}x_{2j-1}, \ \ 1\leq i \leq m.
\end{equation}	
\end{lemma}
\begin{proof} Since on $I_c$ with $c=(c_1, \dots, c_m)$ we have $x_{2i-1}+x_{2i}=c_i,$ each quadratic term in \eqref{uyq} can be simplified as $$x_{2i}x_{2j-1}-x_{2i-1}x_{2j}=(c_i-x_{2i-1})x_{2j-1}-x_{2i-1}(c_j-x_{2j-1})=c_ix_{2j-1}-c_jx_{2i-1}.$$
	Substituting these in (\ref{uyq}) one gets (\ref{chiz}).
\end{proof}
By \eqref{ss} and Lemma \ref{lem2} we conclude that the main problem of the study of limit points for the operator $W$ is reduced to the study of limit points for each linear operator $W_c:I_c\to I_c$ given by (\ref{chiz}).

Consider the matrix $A=(a_{ij})_{i,j=1}^m$, consisting of the coefficients of operator (\ref{uyq}). For a given $c\in S^{m-1}$ to simplify the notation we introduce the matrix
\begin{equation}\label{b} B_c=(b_{ik})_{i,k=1}^m, \quad \mbox{with} \quad b_{ii}=1-\sum_{{j=1\atop j\ne i}}^{m} a_{ij}c_j, \quad
	b_{ik}=c_ia_{ik}, \ \ \mbox{for} \ \ i\ne k.
\end{equation}

Setting $u_i=x_{2i-1}$, we obtain
$$J_c=\{u=(u_1, \dots, u_{m})\in \mathbb R^{m} \ : \ 0\leq u_i\leq c_i, \, \sum_{j=1}^m u_j\leq 1\}.$$
Then $W_c$ can be considered as an operator $V_c$ acting on the coordinates with odd indices, which maps $J_c$ to itself and is given by the relation
$$V_c: \quad u'=B_cu^T,$$
where $u'=(u'_1, \dots, u'_m), \ \ u=(u_1, \dots, u_m)\in J_c.$

Thus, the dynamics of $V_c$ is completely determined by the matrix $B_c$. Specifically, it is sufficient to study the following dynamical system
\begin{equation}\label{lmn}
	 u^{(n+1)}=B_c(u^{(n)})^T, \ \ n\geq 0,
\end{equation}
where $u^{(0)}\in J_c$ is an initial vector.

To answer the main problem of the dynamical system (\ref{lmn}) one has to know the behavior of $B_c^n$ as $n\to \infty$.

Due to associativity property, we have 
\begin{equation}\label{lme}
	u^{(n)}=B_c^n(u^{(0)})^T, \ \ n\geq 0.
\end{equation}

\begin{lemma}\label{Bi} For each $c\in S^{m-1}$, the matrix  $B_c$ has eigenvalue 1.
\end{lemma}
\begin{proof}
	It suffices to show that $B_cc^T=c^T.$ Due to (\ref{b}) for each $i\in\{1,\dots,m\}$ we get
	$$(B_cc^T)_i= \sum_{j=1}^mb_{ij}c_j=b_{ii}c_i+ \sum_{{j=1\atop j\ne i}}^mb_{ij}c_j=\left(1-\sum_{{j=1\atop j\ne i}}^{m} a_{ij}c_j\right)c_i+\sum_{{j=1\atop j\ne i}}^mc_ia_{ij}c_j=c_i.$$
\end{proof}

\begin{lemma}\label{si} Let the matrix $A=(a_{ij})_{i,j=1}^m$ be symmetric (i.e. $a_{ij}=a_{ji}$). Then the matrix $B_c$ is left stochastic, meaning that each of its entries is a non-negative real number and each column sums to 1.
\end{lemma}
\begin{proof} By the definition of $B_c$ we have $0\leq b_{ik}\leq 1$ for $i\ne k$. To compute the sum of the elements in the
$k$-th column (taking into account that $A$ is symmetric), we get
	$$\sum_{i=1}^mb_{ik}=1-\sum_{{j=1\atop j\ne k}}^{m} a_{kj}c_j+
	\sum_{{i=1\atop i\ne k}}^{m} a_{ik}c_i=1.$$
\end{proof}
It is well-known that all eigenvalues of a stochastic matrix have absolute values less than or equal to one.

Since each stochastic matrix defines a Markov chain (see \cite{SK}), from the lemmas proved above we can derive the following result:

\begin{cor} If the condition of Lemma \ref{si} is satisfied, then the dynamical system generated by QSO $W$ is decomposed to uncountable set of Markov chains.
\end{cor}

	\begin{ex} For a given $a,b\in (0,1]$ we consider the case
		\begin{equation}\label{m3} m=3, \, a_{12}=a_{21}=a, \, a_{13}=a_{31}=a_{23}=a_{32}=b.
		\end{equation}
With $c=(\alpha, \beta, \gamma)$ under conditions (\ref{m3}) the matrix $B_c$ has the following form
		$$\left(\begin{array}{ccc}
			1-a\beta-b\gamma& a\alpha & b\alpha\\[2mm]
			a\beta & 1-a\alpha-b\gamma & b\beta\\[2mm]
			b\gamma & b\gamma & 1-b\alpha-b\beta
		\end{array}\right).$$
Evidently, $B_c$ is a left stochastic matrix.
		
Using $\alpha+\beta+\gamma=1$, one can get that the eigenvalues of this matrix are
		$$\lambda_1=1, \ \ \lambda_2=1-b, \ \  \lambda_3=1-a+(a-b)\gamma.$$ 	
	For any values of parameters $a,b\in (0,1]$ and $\gamma\in [0,1]$ we have $0<\lambda_2<1$ and $0<\lambda_3<1$. Moreover, since entries of matrix $B_c$ are strictly positive the dynamical system generated by linear operator of $V_c$ has unique limit point.
	\end{ex}
	
For each $t^{(0)}=(x_1^{(0)},\dots , x_{2m}^{(0)})\in S^{2m-1}$, by the equality (\ref{ss}), there exists a unique $c=c(t^{(0)})\in S^{m-1}$ such that $t^{(0)}\in I_c$.

We set 
$$Z(c)=\{i\in \{1,\dots, m\}: c_i=0\}.$$
By $|A|$ we denote number of elements in the set $A$.

\begin{rk} If $Z(c)\ne \emptyset$ for some $c\in S^{m-1}$, then since
	$$x'_{2i-1}=x'_{2i}=x_{2i-1}=x_{2i}=0 \ \mbox{for any} \ i\in Z(c)$$
	the operator $W$ is reduced on simplex $S^{2(m-|Z(c)|)-1}$.
	Therefore, without loss of generality, one can assume $Z(c)=\emptyset$,
	i.e., $c_i> 0$ for any $i\in\{1,\dots, m\}$.
\end{rk}

Let us recall the Perron-Frobenius (P-F) theorem (see e.g. \cite{Mu}, \cite{Se}):

\begin{thm}\label{PF} ({\rm Perron-Frobenius}) 	Let ${\mathbf P}$ be an $ n \times n$ irreducible non-negative matrix. The following assertions hold:
	\begin{itemize}
		\item[1.] {\rm Spectral radius:}
		The matrix ${\mathbf P}$ has a unique largest positive eigenvalue, called the P-F root, denoted by  $ \lambda_{\text{max}}  $\footnote{This eigenvalue is also called the spectral radius of ${\mathbf P}$.}.
		
		\item[2.] {\rm Simplicity:}	The P-F root  $ \lambda_{\text{max}}  $ is a simple eigenvalue (i.e., its algebraic multiplicity is 1).
		Moreover, any other eigenvalue $\lambda$ of ${\mathbf P}$ satisfies $|\lambda| < \lambda_{\text{max}}$.
		
		\item[3.] {\rm P-F eigenvector:} There exists a corresponding positive eigenvector $ v$ (i.e., all coordinates of $v$ are positive) such that ${\mathbf P} v = \lambda_{\text{max}} v $. This vector is unique up to multiplication by a positive scalar. All eigenvectors for other eigenvalues have at least one positive and one negative coordinate.
	\end{itemize}
\end{thm}

In particular, for a non-negative irreducible matrix that is also a stochastic matrix (i.e., all row sums are 1), the Perron-Frobenius theorem implies:

- The P-F root $\lambda_{\text{max}} = 1.$ This eigenvalue corresponds to the steady-state distribution in Markov chains.

- The corresponding positive eigenvector can be normalized to form a probability distribution, which is the stationary distribution of the Markov chain.\\

\begin{lemma}\label{BO} If $a_{ij}>0$ and $c\in S^{m-1}$ is such that $c_i>0$ for all $1\leq i \leq m$, then all eigenvalues of $B_c$ have absolute values less than or equal to 1, and the eigenvalue 1 is simple.
\end{lemma}
\begin{proof} By Lemma \ref{Bi} we know that $1$ is an eigenvalue of $B_c$.  The corresponding to 1 eigenvector of $B_c$ is $c$, i.e., the unique positive vector mentioned in part 3 of the Perron-Frobenius theorem. Therefore, $\lambda_{\text{max}} = 1$, indeed, if $\lambda_{\text{max}}> 1$ then corresponding positive eigenvector will be different from $c$, which contradicts to part 3 of Theorem \ref{PF}. Hence, $|\lambda|< 1$ for each eigenvalue $\lambda\ne 1$ of $B_c$.  The simplicity of 1 is a consequence of Theorem \ref{PF}.

\end{proof}

Now to study $\lim_{n\to \infty}B_c^n$ we need the following notations and facts (see \cite[Chapter 8]{Ho}):

Let $A$ be a matrix and let $v$ (resp. $w$) be the right (resp. left) eigenvector of $A$ corresponding to the eigenvalue $\lambda$, which mean $Av= \lambda v$ (resp. $w^T A = \lambda w^T$).

Perron-Frobenius theorem (see \cite[Theorem 8.4.4]{Ho}) says that
for $\lambda_{\max}$ both $v$ and $w^T$ are positive, and satisfy
the normalization condition $w^T v = 1$, which  ensures that the eigenvectors are appropriately scaled so that their inner product equals 1.

The matrix $vw^T$ is called Perron projection. It is the projection onto the eigenspace corresponding to the Perron-Frobenius eigenvalue $\lambda_{\max}$. This projection has the following key property:
\begin{equation}\label{PP}
\lim_{k \to \infty} \frac{A^k}{\lambda^k_{\max}} = vw^T.
\end{equation}
Thus, as $k$ increases, the powers of $A$ eventually align along the Perron-Frobenius eigenvalue and its associated eigenspace. All other eigenvalue contributions become negligible in comparison to $\lambda_{\max}$, and the matrix converges to the rank-1 matrix $vw^T$. This means that the long-term behavior of iterations of the matrix $A$ is governed by this projection onto the dominant eigenspace.

In the case of $B_c$ from Lemma \ref{BO}, we have $\lambda_{\max}=1$ and corresponding the right eigenvector is $v=c$.
For $w^T=(w_1, \dots, w_m)^T$ we have
$$w_1c_1+\dots+w_mc_m=1.$$

Therefore, for initial point $u^{(0)}\in J_c$, we find
\begin{equation}\label{oo}
\lim_{n \to \infty} u^{(n)}= \lim_{n \to \infty}B_c^n(u^{(0)})^T=  (cw^T)(u^{(0)})^T= \beta c,	
\end{equation}
where
$ \beta\equiv  \beta(u^{(0)})$ is the coefficient of the Perron projection of  vector $u^{(0)}$ to the one-dimensional space $\{rc: r\in \mathbb R\}.$

This equality (\ref{oo}) can now be used to give limits of all coordinates of trajectory $W^nt^{(0)}$,  $t^{(0)}\in I_c\subset S^{2m-1}$ .

Recall that a point $t\in I_c$ is called {\it a fixed point for $W_c: I_c\to I_c$} if $W_c(t)=t$.

Denote the set of all fixed points by Fix$(W_c)$.

It is easy to see that the set of all fixed points of $W_c$ is
$${\rm Fix}(W_c)=\{(u_1, c_1-u_1, \dots, u_{m}, c_m-u_m)\in I_c \ : \  \sum_{j=1}^{m} a_{ij}\left(c_iu_j-c_ju_i\right)=0\}.$$

\begin{lemma} The set ${\rm Fix}(W_c)$ is an uncountable subset of $I_c$. In particular, $	(c_1, 0, ..., c_{m}, 0)\in {\rm Fix}(W_c)$.
\end{lemma}
\begin{proof} We have to find vectors $u=(u_1, \dots, u_m)$ with the condition $(u_1, c_1-u_1, \dots, u_{m}, c_m-u_m)\in I_c$ such that
	\begin{equation}\label{fix}
		\sum_{j=1}^{m} a_{ij}\left(c_iu_j-c_ju_i\right)=0 \ \mbox{for all} \ i\in \{1, \dots, m\}.
	\end{equation}
	Straightforward computations lead that $(u_1, \dots, u_m)=c$ 
satisfies (\ref{fix}).
	For $m\geq 2$ we set 
	$$H_c = \begin{bmatrix}
		\sum\limits_{j \ne 1} a_{1j} c_j & -c_1 a_{12} & \dots & -c_1 a_{1m} \\
		-c_2 a_{21} & \sum\limits_{j \ne 2} a_{2j} c_j & \dots & -c_2 a_{2m} \\
		\vdots & \vdots & \ddots & \vdots \\
		-c_m a_{m1} & -c_m a_{m2} & \dots & \sum\limits_{j \ne m} a_{mj} c_j
	\end{bmatrix}$$
	Then the system of equations (\ref{fix}) can be rewritten as
	$H_cu^T=0$. Since $u=c$ satisfies $H_cc^T=0$ we have that $\det(H_c)=0$, therefore $H_cu^T=0$ has infinitely many solutions.
\end{proof}

\begin{thm} For any $c\in S^{m-1}$ and $t^{(0)}=(u_1^{(0)}, c_1-u_1^{(0)}, \dots, u_m^{(0)}, c_m-u_m^{(0)})\in I_c\subset S^{2m-1}$  corresponding trajectory generated by non-linear operator (\ref{uyq}) has limit
\begin{equation}\label{D}	\lim_{n\to\infty}W^nt^{(0)}=(\beta c_1, c_1(1- \beta),  \beta c_2, c_2(1- \beta), \dots,  \beta c_m, c_m(1- \beta)),
\end{equation}
where $ \beta= \beta(t^{(0)})$ defined by (\ref{oo}).
\end{thm}
\begin{proof} Using (\ref{oo}) by the equality
$$x^{(n)}_{2i-1}+x^{(n)}_{2i}=u^{(n)}_{i}+x^{(n)}_{2i}=c_i, \, i=1, \dots, m.$$
one completes the proof.
\end{proof}

\section{Conclusions}

The main innovation of this paper is the construction of a family of QSOs, whose dynamics can be reduced to a continuum of linear dynamical systems. Specifically, under certain parameter conditions of the QSO, these linear operators generate Markov chains. We provide a comprehensive description of the limit points of QSO (\ref{uyq}).

From the existence of the limit point of any trajectory and its explicit form of ${\rm Fix}(W)$ it follows that
$$\lim_{n\to \infty}\sum_{j=1}^{m} a_{ij}\left(x^{(n)}_{2i}x^{(n)}_{2j-1}-x^{(n)}_{2i-1}x^{(n)}_{2j}\right)=0.$$
This property, biologically means (see \cite[page 69]{E} for two loci case) that the
population asymptotically goes to a state of linkage equilibrium with respect to many loci.
The linkage equilibrium describes a situation where the alleles (gametes) at different loci are independently associated with one another, meaning that the frequency of a particular combination of alleles can be predicted by the product of the frequencies of the individual alleles.

Biological significance:

- In a state of linkage equilibrium, the alleles at different loci segregate independently of each other.

- Over time, recombination breaks down the associations between alleles at different loci, leading the population towards linkage equilibrium.

- The genotype frequencies at multiple loci can be predicted by the product of the allele frequencies at each locus.

\section*{ Acknowledgements}
The second author thanks the Institute for Advanced Study in Mathematics (IASM) at Harbin Institute of Technology (China) for supporting his visit to IASM.

\end{document}